\documentclass[11pt, reqno]{amsart}

\usepackage{amsmath}
\usepackage{color}

\definecolor{blu}{rgb}{0,0,0.1}

\makeatletter
\def\itemize{
  \ifnum\@itemdepth>3\@toodeep\else
    \advance\@itemdepth\@ne
    \edef\@itemitem{labelitem\romannumeral\the\@itemdepth}%
        \list{\csname\@itemitem\endcsname}%
      {\leftmargin=20pt\def\makelabel##1{\hss\llap{##1}}}
        \fi}

   \makeatletter
   \def\LaTeX{\leavevmode L\raise.42ex
       \hbox{\kern-.3em\size{\sf@size}{0pt}\selectfont A}\kern-.15em\TeX}
   \makeatother
   
   \newcommand{\BibTeX}{{\rm B\kern-.05em{\sc
             i\kern-.025emb}\kern-.08em\TeX}}
\def\bbm[#1]{\mbox{\boldmath $#1$}}



   \newcommand{\e }{\varepsilon }

   \renewcommand{\O }{\Omega }

   \newcommand{\intr }{\int_{\R^3}}
   \newcommand{\into }{\int_{\Omega}}

   \newcommand{\epi}{{\e,P}}

   \newcommand{\dd}{{\rm d}_{\partial\Omega}}
   
   \newcommand{\R}{{\mathbb{R}}}

   \newcommand{\N}{\mathbb{N}}

     \newcommand{\p}{{\bf P}}

\newcommand{\cal}{\mathcal }

   \newcommand{\beq}{\begin{equation}}
   \newcommand{\eeq}{\end{equation}}

   \newtheorem{theorem}{Theorem}[section]
   
   \newtheorem{proposition}[theorem]{Proposition}
   \newtheorem{lemma}[theorem]{Lemma}
   
   \newtheorem{remark}[theorem]{Remark}
   \newcommand{\bremark}{\begin{remark} \em}
   \newcommand{\eremark}{\end{remark} }

\DeclareMathOperator{\di}{d}
\def\bbm[#1]{\mbox{\boldmath $#1$}}
\begin{document}

\title[multiple alternate sign peaks]{
Solutions with  multiple alternate sign peaks along a boundary geodesic to a 
semilinear Dirichlet problem}
\thanks{The    authors  are supported by
  Mi.U.R.   project ``Metodi variazionali e topologici nello studio
  di fenomeni non lineari''.}

\author{Teresa D'Aprile \& Angela Pistoia}
\address{Dipartimento di Matematica, Universit\`a di Roma ``Tor
Vergata", via della Ricerca Scientifica 1, 00133 Roma, Italy.}
\email{daprile@mat.uniroma2.it }
\address{Angela Pistoia, Dipartimento SBAI, Universit\`a di Roma
``La Sapienza", via Antonio Scarpa 165, 00161 Roma, Italy.}
\email{pistoia@dmmm.uniroma1.it }

\begin{abstract}We study the existence  of  sign-changing multiple interior spike
solutions for the following Dirichlet problem \begin{equation*}\e^2\Delta v-v+f(v)=0\hbox{ in }\Omega,\quad v=0
\hbox{ on }\partial \Omega,\end{equation*} where $\Omega $ is a smooth and bounded domain of
$\R^N$, $\e$ is a small positive parameter, $f$ is a superlinear, subcritical and odd nonlinearity. 
In particular we prove  that if $\Omega$ has a plane of symmetry and its intersection with the plane is a two-dimensional strictly convex domain, then, provided that $k$ is even and sufficiently large, a $k$-peak solution exists with alternate sign    peaks aligned along a closed curve near a geodesic of $\partial \Omega$.
\end{abstract}
\maketitle
\section{Introduction}
The present paper is concerned with  the following singularly perturbed elliptic problem:
\begin{equation}\label{eq1}\left\{\begin{aligned}&\e^2\Delta v-v+|v|^{p-2}v=0&\hbox{ in }&\Omega,\\ &v=0
&\hbox{ on }&\partial \Omega,\end{aligned}\right.\end{equation} where $\Omega $ is a smooth and bounded domain of
$\R^N$, $N\geq 2$, $2<p<\frac{2N}{N-2}$ if $N\geq 3$ and $p>2$ if $N=2$, and $\e>0$ is a small parameter.

This problem arises from different mathematical models: for instance,
it appears in the study of stationary solutions for the Keller-Segal system in chemotaxis and the Gierer-Meinhardt system in biological pattern formation.

In the pioneering paper   \cite{nw} Ni and Wei proved that for $\e>0$  sufficiently small problem
\eqref{eq1} has a positive   least energy solution $v_\e$  which develops a spike layer at the {\em most
centered} part of the domain, i.e. $ {\di}_{\partial\Omega}(P_\e) \to \max_{P \in \Omega} {\di}_{\partial\Omega}(P)$,
where $P_\e$ is the unique maximum of $v_\e$.  Hereafter ${\di}_{\partial \Omega}(P)$ denotes the distance of $P$
from $\partial\Omega$.
Since then, there have been many works looking for  positive  solutions with single and multiple
 peaks and investigating  the location of the asymptotic  spikes as well as 
their profile  as $\e\to 0^+$.
More specifically, several papers study the effect of the geometry of the domain on the
existence of positive $k-$peak solutions  (see \cite{bc,cdny,d1,d2,dy0,dy,dw,dfw1,dfw2,gp,ln,w1} and references therein).
In particular,  Dancer and Yan (\cite{dy0}) proved that if the domain has a nontrivial topology, then there always exists a  $k$-peak positive solution for any $k\geq 1$. This result has been generalized by 
Dancer, Hillman and Pistoia (\cite{dhp})  to the case of a not contractible domain.   On the other hand, Dancer and Yan
 (\cite{dy})   showed that if $\Omega$ is a strictly convex domain  and $k\geq 2$, then  problem \eqref{eq1} does
not admit a $k$-peak positive solutions   (see also \cite{wei0} for the proof when $k=2$).

The first result concerning existence of sign changing solutions was obtained by Noussair and Wei  (\cite{nouwei1}). They
  proved that for $\e$ sufficiently small \eqref{eq1} has a least energy nodal solution
  with one positive   and one negative peak centered at points $P_1^\varepsilon$, $P^\e_2$ whose location depends on the geometry of the domain $\Omega$.
  More precisely, if $\bar P_1$, $\bar P_2$ are the limits of a subsequence of $P_1^\e$, $P_2^\e$,  respectively,   then  $(\bar P_1, \bar P_2)$ maximizes the function 
$$\min \left\{ {\di}_{\partial\Omega}(P_1),\ {\di}_{\partial\Omega}(P_2),\
{|P_1-P_2|\over2}\right\},\quad P_1,\, P_2\in\Omega\times \Omega.$$  Moreover, Wei and Winter (\cite{ww}) showed that  such solution is odd  in one direction when $\Omega$ is the unit ball.
 Successively, Bartsch and Weth in \cite{bawe1,bawe2}, by using a different approach, found  a lower bound on the
number of sign-changing solutions. These papers are however not concerned with the shape of the solutions.

As far as we know the question of the existence of $k$-peaked nodal   solutions for problem
\eqref{eq1} for any  $k\geq 3$    is largely open. In a general domain, D'Aprile and Pistoia in \cite{dapi1} constructed solutions
with $h$ positive peaks and $k$ negative peaks as long as $h+k\le6$.  They also found
solutions with an arbitrarily large number of mixed positive and negative   peaks  provided some symmetric assumptions are satisfied: in the case 
of a domain  $\Omega$ symmetric with respect to a line, where the peaks are aligned with alternate   sign  along the axis of symmetry, and in the case of a ball, where the   peaks are located with alternate sign at the vertices of a  regular polygon with an even number of edges.

We believe that it should be possible to extend the above results to a more general domain.
More precisely, we conjecture that

{\em 
\begin{itemize}
\item[(C1)]   there exists a solution with alternate sign peaks aligned on an interior straightline intersecting with $\partial\Omega $ orthogonally;
\item[(C2)] there exists a solution with alternate sign peaks aligned  
on a curve close to a closed geodesic  of $\partial\Omega.$ 
\end{itemize}
}

 In the present paper,  we prove that the conjecture (C2) is true at least when $\Omega$ has a plane of symmetry and its intersection with the plane is a two dimensional strictly convex domain   (see the assumptions  (a1), (a2) below).

\medskip
In order to provide the exact formulation of the main result let us
fix some notation. We point out that most of the results contained
in the aforementioned papers can be extended to equations where
$|v|^{p-2}v$ is replaced by a more general nonlinear term. Then we
will consider the more general problem
\begin{equation}\label{sch}
\left\{\begin{aligned}&\e^2\Delta v- v+f(v)=0 &\hbox{ in }&\Omega,\\ &v=0 &\hbox{ on }&\partial
\Omega.\end{aligned}\right.\end{equation} We will assume that $f:\R\rightarrow\R$ is of class ${\mathcal C}^{1+\sigma}$ for some $\sigma>0$
and satisfies the following conditions:
\begin{itemize}
\item[(f1)] $f(0)=f'(0)=0$ and $f(t)=-f(-t)$ for any $t\in\R$;
\item[(f2)] $f(t)\to +\infty$, $f(t)=O\left(t^{p_1}\right)$, $f'(t)=O\left(t^{p_2-1}\right)$ as $t\rightarrow+\infty$ for some $p_1,p_2>1$ and
there exists $p_3>1$ such that
$$\forall s,t:\;\;\left|f'(t+s)-f'(t)\right|\le\left\{
\begin{aligned}
 &c|s|^{ p_3 -1} &\ \hbox{if}\ p_3>2\\
 &c\left(|s|+|s|^{ p_3 -1}\right) &\ \hbox{if}\ p_3\le2\\
  \end{aligned}
\right.$$ for a suitable $c>0$;
\item[(f3)] the following problem
\begin{equation*}
 \left\{
\begin{aligned}
 &\Delta w -w + f(w)=0, \;\;w>0\;\; \mbox{ in } \R^N \\
  &w(0)=\max_{z \in \R^N} w(z),\;\; \lim_{ |z| \to +\infty} w(z)
  =0
  \end{aligned}
\right.
\end{equation*} has a unique solution $w$ and $w$  is nondegenerate, namely the linearized operator
$$L:H^2(\R^N )\to L^2(\R^N ),\;\; L[u]:=\Delta u -u  + f' (w)u,$$
satisfies
\begin{equation*}
 \mbox{Kernel} (L) = \mbox{span} \left\{ \frac{\partial w}{\partial z_1},\ldots, \frac{\partial
w}{\partial z_{N }} \right\}.
\end{equation*}
\end{itemize}
By the well-known result of Gidas, Ni and Nirenberg (\cite{gnn}) $w$ is radially symmetric and strictly
decreasing in $r=|z|$. Moreover, by classical regularity arguments, the following asymptotic result holds
\begin{equation}\label{wdecay}
  \lim\limits_{|z|\rightarrow+\infty}|z|^{N-1\over2}e^{|z|}w(z)=A>0\;\ \hbox{ and }\;\
\lim\limits_{|z|\rightarrow+\infty} {w'(z)\over w(z) }=-1.
\end{equation}
 The class of nonlinearities $f$ satisfying (f1)-(f3) includes,
and it is not restricted to, the model $f(v)=|v|^{p-2}v $ with
$p>2$ if $N=1,2$ and $2<p<\frac{2N}{N-2}$ if $N\geq 3$.
 Other nonlinearities can be found in \cite{dancer}.

Here are our assumptions on $\Omega$.
\begin{itemize}
\item[(a1)] $\Omega$ is  a
bounded domain with a ${\cal C}^2$ boundary, symmetric with respect to the $x_i$'s axes for $i=3,\dots,N,$ i.e.
$$(x_1,\dots,x_i,\dots,x_N)\in\Omega\  \Leftrightarrow\ (x_1,\dots,-x_i,\dots,x_N)\in\Omega\qquad\forall i=3,\ldots, N;$$
\item[(a2)] the relative boundary of $\Omega_0:=\Omega\cap\{x\in \R^N\ :\ x_3=\dots=x_N=0\}$ has a 
connected component $\Gamma$ satisfying
$$\nu_P\cdot (P-Q)>0\quad \forall P,Q\in \Gamma, \; P\neq Q$$ where $\nu_P$ is the unit outward normal to $\partial \Omega$ at $P$.
\end{itemize}
It is clear that if $\Omega$ is a two-dimensional  strictly convex domain, the above  assumptions are automatically satisfied and,  in particular, $\Gamma$ coincides with the exterior boundary of $\Omega.$
 More in general, we point out that
if  $N\ge 3,$ then
  $\Gamma$ turns out to be a closed geodesic of $\partial\Omega.$

 The main purpose of this paper is to prove that if $\Omega$ satisfies (a1), (a2) then, provided that $\delta$ is sufficiently small and $k$ is even and sufficiently large,  the problem \eqref{sch} admits a $k$-peak  solution  with $k$ alternate sign peaks
 aligned near
 $\Gamma$.
More precisely, the limiting configuration
can be described in the following way:   the $ k$  peaks lie in $\Omega_0$ and are arranged with alternate sign  at distance $\delta$ from
 $\Gamma$ and
 the distance between
 two consecutive peaks  is $2\delta$. Roughly speaking, the limit profile of such solution resembles a crown of peaks surrounding $\Gamma$.
Moreover the profile of each peak is similar to
a translation
of the rescaled ground state $w$.
Now we proceed to provide the exact formulation of the result.

\begin{theorem}
\label{th1}  Assume that hypotheses  (f1), (f2) and (f3) and (a1), (a2) hold.
 Then for any
$\delta_0>0$  there exist $\delta \in(0,\delta_0)$  and     an even  integer $k $ such that,
for $\e$ sufficiently small, the
problem \eqref{sch} has a  solution $v_\e\in H^2(\Omega)\cap H^1_0(\Omega)$ symmetric with respect to the $x_i$'s axes for $i=3,\dots,N,$ i.e.
$$v_\e(x_1,\dots,x_i,\dots,x_N)=v_\e(x_1,\dots,-x_i,\dots,x_N)\qquad\forall i=3,\ldots, N.$$
 Furthermore there exist points $P_{1}^\e,\ldots,
P_k^\e\in\Omega_0$ such that, as $\e\to  0^+$,

 \begin{equation}\label{deluxe}v_\e (x) = \sum_{i=1}^k(-1)^i w\Big(\frac{x-P_{i}^\e}{\e}
\Big)+ o(e^{-\frac{\delta}{2\e}})\hbox{ uniformly for } x\in\overline{\Omega}.\end{equation}
Moreover, if $(P_1^\ast, \ldots, P_{k}^\ast)$ is the limit of a subsequence of
 $( P_{1}^{\e},\ldots, P_{k}^{\e})$ as $\e\to 0^+$, then\footnote{The relation ``$P_i^*<P_{i+1}^*$" refers to a cyclic order on the closed curve $\{P\in \Omega_0\,|\,  {\di}_{\Gamma}(P)=\delta\}.$}
 \begin{equation}\label{pack}  {\di}_{\Gamma}(P^*_i)=\delta,\;\; P_i^*<P_{i+1}^*,\;\;|P_i^*-P_{i+1}^*|=2\delta \quad \forall i=1,\ldots, k \quad (P_{k+1}^*:=P_1^*).\end{equation}
\end{theorem}

\bigskip

 The assumption that $\Omega$ has a plane of symmetry allows to locate the points where the spikes occur along a curve in the plane. Indeed,  in the general case the problem of packing the spikes near $\partial \Omega$ in equilibrium is not so simple.  We believe that more complicated arrangements should exist depending
on the geometry of $\partial \Omega$. In particular, as we mentioned above, we conjecture that a possible balanced pattern  may occur for spikes tightly aligned near a closed geodetics of $\partial\Omega$.

We now outline the main idea of the proof of Theorem \ref{th1}.

As with many of the other results mentioned above, a Lyapunov-Schmidt reduction scheme is used in the vicinity of multi-peaked approximate solutions.  A sketch of this procedure is
given in  Section 2. By carrying out the reduction process,  we reduce the problem of finding multiple interior spike solutions for \eqref{sch} to the problem of finding critical points of a vector field on the finite dimensional manifold  consisting of multi-spike states. More precisely, in order to find such a  solution the limiting location of the spikes
should be critical for a functional of this type
\begin{equation}\label{redu1}\sum_{i=1}^ke^{-\frac{2{\di}_{\partial\Omega}(P_i)+o(1)}{\e}}
 -\sum_{i,j=1\atop i< j}^k(-1)^{i+j}e^{-\frac{|P_i-P_j|+o(1)}{\e}}+h.o.t.\end{equation}
on a suitable configuration set in $ \Omega^k$.
The terms $e^{-\frac{2{\di}_{\partial\Omega}(P_i)+o(1)}{\e}}$ represent the boundary
effect on each  spike $P_i$, created by the boundary condition, while the terms $e^{-\frac{|P_i-P_j|+o(1)}{\e}}$ are due to the interaction among the peaks which has an
attractive or a repulsive effect according to their respective sign.  The presence of a factor 2 in the exponentials $e^{-\frac{2{\di}_{\partial\Omega}(P_i)+o(1)}{\e}}$
suggests that  the effect of the boundary  acts exactly as
an opposite \textit{virtual} peak reflected in $\partial\Omega$.
Moreover the setting of Theorem \ref{th1} suggests that we should restrict ourselves to seeking equilibrium points $P_i\in \Omega_0$, i.e. 
$$P_i=(\xi_i,{\bf 0}), \quad \xi_i\in \R^2,\;{\bf 0}=(0,\ldots, 0)\in \R^{N-2}.$$ 
The different interaction effects of the boundary and the peaks, which
depend upon their distance in an exponential way,
provide
the functional described by \eqref{redu1} with a suitable local minimum structure.

To give an idea how to apply  the minimization argument,
let us make the following heuristic considerations.
If we look for equilibrium points  $P_1,\ldots,P_k$ which are packed in a tight strip near $\Gamma$ in such a way that the neighboring spikes have opposite sign, then the peaks having the same signs are non-interacting to the leading order, and this implies that the terms $e^{-\frac{|P_i-P_j|+o(1)}{\e}}$ do not contribute to the main term of the expansion if $(-1)^{i+j}=1$, so  \eqref{redu1} actually equals
\beq\label{redu2}\sum_{i=1}^ke^{-\frac{2{\di}_{\partial\Omega}(P_i)+o(1)}{\e}}
 +\sum_{i,j=1\atop i< j}^ke^{-\frac{|P_i-P_j|+o(1)}{\e}}+h.o.t..\eeq  
By \eqref{redu2} we get that the spikes $P_1,\ldots, P_k$  are repelled  from $\partial \Omega$ and doubly from one another, then they may exist in equilibrium when they are packed exactly  as in \eqref{pack}: indeed the arrangement \eqref{pack} assures that the nonvanishing forces exerted by the boundary and the neighboring spikes balance giving rise to the equilibrium configuration $(P_1^*,\ldots,P_k^*)$.

The paper is organized as follows.
Section 2 contains the reduction to the finite dimensional problem, which is done by using the Lyapunov-Schmidt decomposition at the approximate solutions.
In Section 3 we study a minimization problem which provides the \textit{equilibrium} arrangement of the alternate sign spikes around $\Gamma$; then we show that the solution of the minimization problem is indeed associated to a  solution of \eqref{sch} which satisfies all the properties of Theorem \ref{th1}.

\section{The reduction process: sketch of the proof}
In this section we outline the main steps of the so called
{\textit{finite dimensional reduction}, which reduces the problem
to finding a critical point for a functional on a finite
dimensional space. Since this procedure is carried out in a standard way, we omit the proofs and refer to  \cite{cw,gpw,gw,nw} for technical details.

First we introduce some notation and present some important estimates on the approximate solutions.
Associated with problem \eqref{sch} is the following energy functional
$$J_\e(v)=\frac12\into \big(\e^2|\nabla v|^2+v^2\big) dx-\into F(v) dx ,\quad v\in H^1_0(\Omega),$$
where $F(t)=\int_0^t f(s) ds$.

For $P\in\Omega$ let $w_\epi$  be the unique solution of
$$\left\{\begin{aligned}&\e^2\Delta v-v +f\Big(w\Big(\frac{x-P}{\e}\Big)\Big)=0&\hbox{ in }& \Omega,\\ & v=0&\hbox{ on }&\partial \Omega.\end{aligned}\right.$$$w_\epi$ is a kind of projection of $w(\frac{x-P}{\e})$ onto the space $ H^1_0(\Omega)$.

Then, if we set
$$\psi_{\e,P}(x):=-\e\log \Big(w\Big(\frac{x-P}{\e}\Big)-w_\epi\Big),\quad \psi_\e(P)=\psi_{\e,P}(P),$$
it is well known that\begin{equation}\label{known0}\psi_{\e,P}(x)\to \inf_{z\in\partial\Omega}\{|z-x|+|z-P|\}
\end{equation}  and, consequently,
\begin{equation}\label{known}\psi_\e(P)\to 2\dd(P)
\end{equation}
uniformly for $x\in \overline{\Omega}$ and $P$ on compact subsets of $\Omega$
(see \cite{nw}, for instance).

Fixed $k\geq 1$, we define the configuration space
$$\Lambda_\eta:=\Big\{(P_1,\ldots,P_k)\in \Omega^{k}\,\Big|\, \di_{\partial\Omega}(P_i)>\eta\;\,\forall i,\;\; |P_i-P_j|>\eta\hbox{ for }i\neq j \Big\}$$
where $\eta>0$ is  a sufficiently small number.
For $\p=(P_1,\ldots, P_k)\in
\Lambda_\eta$ we set
$$w_{\e,\p}=\sum_{i=1}^k(-1)^i w_{\e, P_i}.$$

We look for a solution to \eqref{sch} in a small neighborhood of
the first approximation $ w_{\e,\p}$, i.e. solutions of the form as
$v:= w_{\e,\p}+\phi,$ where the rest term $\phi$ is \textit{small}.
To this aim, for $v\in H^2(\Omega) $ we put
$${\cal S}_\e[v]=\e^2\Delta v-v+f(v).$$Then the problem \eqref{sch} is equivalent to solve
$${\cal S}_\e[v]=0, \quad v\in H^2(\Omega)\cap H^1_0(\Omega).$$
 We introduce the following approximate cokernel and kernel
$$\begin{aligned}
& {\mathcal{K}}_{\e,\p}={\rm span}\left\{{\partial w_{\e,\p}\over \partial P_i^l}\ :\ i=1,\dots,k,\ l=1,\dots,N\right\}\subset { H}^2(\Omega)\cap H^1_0(\Omega),\\
& {\mathcal{C}}_{\e,\p}={\rm span}\left\{{\partial w_{\e,\p}\over \partial P_i^l}\ :\ i=1,\dots,k,\ l=1,\dots,N\right\}\subset { L}^2(\Omega),\\
\end{aligned}$$
 denoting by $P_{i}^l$ the $l$-th component of
$P_i$ for $l=1,\ldots, N$.
The idea is that we first solve $\phi=\phi_{\e,\p}$ in  ${\mathcal{K}}_{\e,\p}^\perp$, where the orthogonal is taken with respect to the scalar product in $H^1_0(\Omega):$
$$\langle u, v \rangle_\e=\into\big( \e^2\nabla u\nabla v+uv\big)\,dx.$$
The following lemma is proved in \cite{cw,gw}.
  \begin{lemma}\label{reg}
Provided that $\e>0$ is sufficiently small,
   for every    $ {\bf P }\in
   \Lambda_\eta$ there exists a unique  $ \phi_{\e,\bf P }\in {\mathcal{K}}_{\e,\p}^\perp$ such that
   \beq\label{sati1}{\cal S}_\e [w_{\e,\p}+\phi_{\e,\p}]\in {\mathcal{C}}_{\e,\p}.\eeq
 Moreover
 the map $\p\in
 \Lambda_\eta\mapsto \phi_{\e,\bf P}\in H^1_0(\Omega)$ is ${\mathcal C}^1$  and
 \beq\label{sati2}|\phi_{\e,\p}|\le C\e^{-\left(1+{\sigma\over2}\right){\varphi_{k}(\p)\over\e}}\eeq
 where the function  $\varphi_k:\Omega^k\rightarrow\R$  is defined by
\begin{equation*}
\varphi_k(\p):=\min\limits_{i,j=1,\dots,k\atop i\not =j}\left\{ {\di}_{\Gamma}(P_i),\
{|P_i-P_j|\over2}\right\},\quad \p:=(P_1,\dots,P_k) . \end{equation*}
   \end{lemma}
\bigskip

   After that, we define a new functional:
\begin{equation*}
\label{Mept}  M_\e:
 \Lambda_\eta\to\R,\;\; M_\e[{\bf P }
]:= \frac{\e^{-N}}{\gamma}J_\e [  w_{\e,\p} + \phi_{\e,{\bf P }}] - \frac{c_1}{\gamma}
\end{equation*}
 where $ \phi_{\e,\bf P }$ has been constructed in Lemma \ref{reg} and
 $$c_1=\frac{k}{2}\intr |\nabla  w|^2dx-k\intr F(w) dx,\quad\gamma=\intr f(w)e^{x_1} dx.$$
Next proposition  contains the key expansion of $M_\e$
(see \cite{gw} for the proof).
\begin{proposition}\label{exp1}
The following asymptotic expansions  hold:
\begin{equation}\label{lonel}\begin{aligned}M_\e [ \p]=&\, \frac{1}{2}(1+o(1))\sum_{i=1}^{k}e^{-\frac{\psi_\e(P_i)}{\e}}-(1+o(1))\sum_{i,j=1\atop i<
j}^k(-1)^{i+j}w\Big(\frac{P_i-P_j}{\e}\Big)\end{aligned}
\end{equation}
uniformly for ${\bf P }=(P_1,\ldots,P_k)\in
\Lambda_\eta$.

\end{proposition}

\begin{remark}\label{energy3} By using \eqref{known} and \eqref{wdecay} the expansion \eqref{lonel} can be rewritten as
\begin{equation*}\begin{aligned}M_\e [\p]=&\, \sum_{i=1}^{k}e^{-\frac{2\dd(P_i)+o(1)}{\e}}-\sum_{i,j=1\atop i<
j}^k(-1)^{i+j}e^{-\frac{|P_i-P_j|+o(1)}{\e}}
\end{aligned}
\end{equation*}
uniformly for ${\bf P }=(P_1,\ldots,P_k)\in
\Lambda_\eta$.
\end{remark}

\noindent Finally the next lemma  concerns the relation between
the critical points of $M_\e $ and those of $J_\e $. It is quite
standard in singular perturbation theory; its proof can be found
in \cite{gw}, for instance.

\begin{lemma}\label{relation}
Let ${\bf P }_\e  \in \Lambda_\eta$ be a critical point of $M_\e $.
Then, provided that $\e>0$ is sufficiently small, the
corresponding function $ v_\e= w_{\e,\p_\e} + \phi_{\e, {\bf P
}_\e}$ is  a solution of (\ref{sch}). \end{lemma}

We finish this section with a symmetry property of the reduction process.

\begin{lemma}\label{symmetry}
Suppose $\O$ is invariant under the action of an orthogonal transformation $T\in {\cal O}(N)$. Let
$\Lambda_\eta^T:=\{{\bf P}\in\Lambda_\eta: TP_i=P_i\ \ \forall i\}$ denote
the fixed point set of $T$ in $\Lambda_\eta$. Then a point
${\bf P}\in\Lambda_\eta^T$ is a critical point of $J_\e$ if it is a
critical point of the constrained functional $J_\e|\Lambda_\eta^T$.
\end{lemma}

\begin{proof}
We first investigate the symmetry inherited by the function
$\phi_{\e,\hbox{\scriptsize$\p$}}$ obtained in Lemma \ref{reg}.
Setting $T{\bf P}:=(TP_1,\ldots,TP_k)$ for ${\bf P}=(P_1,\ldots,P_k)\in\O^k$,
we claim that
\begin{equation}\label{cell}
\phi_{\e,{\bf P}}
 =\phi_{\e,T{\bf P}}\circ T\quad
 \forall {\bf P}\in \Lambda_\eta.
\end{equation}
Indeed, because of the symmetry of the domain, we see that
$$
w_{\e,P_i}= w_{\e, TP_i}\circ T
$$
and
$$
{\cal K}_{\e,{\bf P}}
 =\{f\circ T\,|\, f\in {\cal K}_{\e,T{\bf P}}\},\qquad
{\cal K}_{\e,{\bf P}}^\perp
 =\{f\circ T\,|\, f\in {\cal K}^\perp_{\e,T{\bf P}}\}.
$$
Then the function  $\phi_{\e,T{\bf P}}\circ T$ belongs to
$ {\cal K}_{\e,{\bf P }}^\perp$ and satisfies
\eqref{sati1} and \eqref{sati2}. The uniqueness of the solution $\phi$ implies \eqref{cell}.
Therefore the functional $J_\e$ satisfies
$$
J_\e({\bf P})=J_\e(T{\bf P}).
$$
The lemma follows immediately.
\end{proof}

\section{A minimization problem}

In this section we will employ the reduction approach to construct the solutions stated in
Theorem \ref{th1}. The results obtained in the previous section imply that our problem reduces
to the study of critical points of the functional $M_\e$. In what follows, we assume assumptions (a1), (a2). We get the following result.
\begin{lemma}\label{validity}
If $\p$ is a critical point of ${M}_\e|_{\Omega_0}$, then
${\bf P}$ is a critical point of $M_\e$.
\end{lemma}

\begin{proof} This is an immediate consequence of Lemma \ref{symmetry}.
\end{proof}

>From Lemma \ref{validity}, we need to find a critical
point of the functional ${M}_\e|_{\Omega_0}$.

Now we set up a maximization problem for the function $\varphi_k$ defined by 

\begin{equation*}
\varphi_k(\p):=\min\limits_{i,j=1,\dots,k\atop i\not =j}\left\{ {\di}_{\partial\Omega}(P_i),\
{|P_i-P_j|\over2}\right\},\quad \p:=(P_1,\dots,P_k) . \end{equation*}
The function $\varphi_k$ appears naturally in the location of the asymptotic spikes, as we will see at the end of the proof of Theorem \ref{th1}.

First we need some auxiliary lemmas.

\begin{lemma} \label{lem1} Let $\Gamma$ be as in (a2). For any
$\delta_0>0$  there exist $\delta\in (0,\delta_0)$ and an even integer $k$
 such that
$$\sup\left\{\varphi_k(\p)\,\big|\, P_i\in \gamma_\delta,\; P_i<P_{i+1}\;\;(P_{k+1}:=P_1)\right\}
=\delta$$ where $$\gamma_\delta:=\{P\in\Omega_0\,|\, {\rm d}_{\Gamma}(P)=\delta\}.$$ Moreover, if $\p^*=(P^*_1,\ldots, P^*_{{k}})\in(\gamma_\delta)^{ k}$ is such that $\varphi_{ k}(\p^*)=\delta$, then the points  $P_1^*,\ldots, P_k^*$ satisfy \eqref{pack},
i.e. they form a polygonal having
vertices on $\gamma_{\delta}$ and edge $2\delta$.
\end{lemma}

\begin{proof}
The strict convexity of $\Gamma$ implies that, if $\delta_0>0$  is sufficiently small, then  for any $\delta\in(0,\delta_0]$ we have that  $\gamma_\delta$
is a  regular closed curve   and \begin{equation}\label{mimim}\hbox{every point of }\gamma_\delta
  \hbox{ has exactly two points on }\gamma_\delta \hbox{ at distance } 2\delta.\end{equation}
  Then choose $k\in\N$ such that $k$ is even and satisfy $$ k>\frac{\ell(\Gamma)}{2\delta_0},$$
where $\ell(\Gamma)$  denotes the length of the curve $\Gamma$.  We define
 $$\Sigma=\left\{\delta\in (0,\delta_0]\,\Big| \,\exists P_1,\ldots ,P_k\in \gamma_\delta\hbox{ s.t. }P_i<P_{i+1},\;\;|P_i-P_{i+1}|\geq 2\delta\right\}.$$
 The definition of $k$ implies that $\delta_0\not\in \Sigma$.
On the other hand it is easy to prove that  $\Sigma$ contains $\delta$ if $\delta<<\delta_0.$ Let us define $\delta^\ast$ as
$$\delta^*=\sup\{\delta\,|\, \delta\in\Sigma\}.$$
 A straightforward computation shows that $\delta^*$ is actually a maximum (it is sufficient to consider a maximizing sequence and then pass to the limit for a convergent  subsequence), hence $\delta^*\in (0,\delta_0).$
We point out that $\di_{\partial\Omega}(P)=\di_\Gamma(P)=\delta$ for any $P\in\gamma_\delta$ provided that $\delta$ is small enough, so we clearly have
$$\sup_{\p\in (\gamma_{\delta^*}\!)^k}\varphi_k(\p)= \delta^*.$$
Let $\p^*=(P_1^*,\ldots,P_{ k}^*)\in (\gamma_{\delta^*}\!)^{k}$ be such that
 $\varphi_{ k}(\p^*)= \delta^*$, which implies  \begin{equation*}\min_{i\neq j}|{ P}_i^*-{P}_j^*|\geq 2\delta^*.\end{equation*}
We claim that \begin{equation}\label{cicchi}|{P}_{i+1}^*-{P_i}^*|=2\delta^*\quad \forall i=1,\ldots, k.\end{equation} Indeed, assume by contradiction that $|{P}_2^*-{P}_1^*|>2\delta^*$. Then, using \eqref{mimim}, we can  move the points $P_i^*$'s slightly backwards into new points $P_i$'s:
 $$P_1=P_1^*,\;\;\; P_i<P_i^*<P_{i+1} \;\;\;\forall i=2,\ldots, k,
 $$ and the $P_i$'s  verify
 $$|P_{i+1}-P_i|>  2\delta^*\quad \forall i=1,\ldots,k.$$
Consider $\bar P_i$ the projection of $P_i$ onto $\gamma_{\delta}$: by continuity, if $\delta>\delta^* $ is sufficiently closed to $\delta^*$, the $\bar P_i$'s  satisfy
 $$|\bar P_{i+1}-\bar P_i|> 2\delta\quad \forall i=1,\ldots,k$$ which contradicts the maximality of $\delta^*$.

Then \eqref{cicchi} holds, which implies that the $P^\ast_i$'s satisfy \eqref{pack}.

\end{proof}

\begin{lemma}\label{lem20} Let $D\subset\R^N$ be a strictly convex domain.  
Then, for any $\delta>0$ there exists $\eta>0$ such that, if $P,Q\in\partial D$, $|P-Q|\geq \delta$ and
$\eta_1,\,\eta_2\in [0,\eta]$, $(\eta_1,\eta_2)\not=(0,0)$, then
$$|P-\nu_P\eta_1-Q+\nu_Q\eta_2|<|P-Q|,$$
where $\nu_P$ is the unit outward normal to $\partial D$ at $P$.
\end{lemma}

\begin{proof} Fixed $\delta>0$, by the strict convexity  we get
$$\inf_{P\in\partial D,\,|Q-P|\geq \delta}\nu_P\cdot (P-Q)=\eta>0.$$
 For $|P-Q|\geq \delta$, $\eta_1,\,\eta_2\in
[0,\eta]$ with $(\eta_1,\eta_2)\not=(0,0)$, we compute
$$\begin{aligned}|P-\nu_P\eta_1-Q+\nu_Q\eta_2|^2-|P-Q|^2&=2\eta_1(Q-P) \cdot\nu_P+2\eta_2(P-Q) \cdot\nu_Q\\ &\;\;\;+\eta_1^2+\eta_2^2-2\eta_1\eta_2\nu_P\nu_Q\\ &\leq -2(\eta_1+\eta_2)\eta +(\eta_1+\eta_2)^2<0.\end{aligned}$$
\end{proof}

With the help of the previous two lemmas we can now give the following result which will be crucial for the asymptotic locations of the $k$ spikes in the solutions of Theorem \ref{th1}.

\begin{proposition}\label{lemderlem2} Assume that $\Omega$ satisfies (a1)-(a2).  For any $\delta_0>0$ there  exist $\delta\in (0,\delta_0)$ and an even integer $k$  such that, if $\eta$ is sufficiently small, then
\begin{equation}\label{supmax}\sup_{\p\in \partial U_{\eta}}\varphi_{ k}(\bbm[\p])<\sup_{\p\in U_{\eta}}\varphi_{ k}(\bbm[\p])=\delta
\end{equation}
where
 $$U_{\eta}=\{\p\in\Omega_0^k\,|\, \delta-\eta<{\rm d}_{\Gamma}(P_i)<\delta+\eta,\; \bar P_i<\bar P_{i+1}\;\,\forall i,\;\;|P_i-P_j|>2\delta-\eta\;\hbox{ for }i\neq j\}.$$
 Here $\bar P$ denotes the projection of $P$ onto  the curve  $\gamma_{\delta}:=\{P\in\Omega_0\,|\, {\rm d}_{\Gamma}(P)=\delta\}$.

Moreover if $\p^*=({P}^*_1,\ldots, {P}^*_{{k}})\in U_\eta$
is such that  $\varphi_{ k}(\p^*)=\delta$, then the points $P^*_1,\ldots, P_k^*$ satisfy \eqref{pack}.

\end{proposition}

\begin{proof}

Let $\delta\in (0,\delta_0)$ and $k\in\N$ even be such that Lemma \ref{lem1} holds.  Let $D$ be the strictly convex bounded flat domain whose boundary is
   $ \gamma_{\delta}  ,$ which is  contained in $\Omega_0$.    According to Lemma \ref{lem20}, if  $\eta\in
(0, \frac{\delta}{2})$ is sufficiently small, then, for any  $Q,\,Q'\in\gamma_{\delta}$ with  $|Q-Q'|\geq
\frac{\delta}{2}$ and any $\eta_1,\,\eta_2\in [0,\eta]$, $(\eta_1,\eta_2)\not=(0,0)$,  we get
\begin{equation}\label{derlem2}|Q-\eta_1\nu_Q-Q'+\eta_2\nu_{Q'}|<|Q-Q'|.\end{equation}
We are going to prove that, for such $\eta$, \begin{equation}\label{thesis}\p\in\partial U_\eta\quad\Longrightarrow\quad\varphi_{ k}(\p)<\delta. \end{equation}

It is useful to point out that  for any $P\in\partial U_\eta$ we have
$ {\rm d}_{\Gamma}(P)={\rm d}_{\partial\Omega}(P)$, provided  $\delta$ is small enough. Then
  it is immediate that, if ${\rm d}_{\Gamma}(P_i)<\delta$ for some $i$ or $|P_i-P_j|<2\delta$
  for some $i\neq j$, then $\varphi_{k}(\p)<\delta$.  Moreover, if $\bar P_i=\bar P_{i+1}$ for some $i$, then, by construction $|P_i-P_{i+1}|\leq 2\eta<\delta$, and again we get
   $\varphi_{k}(\p)<\delta$.
  Therefore,
 without loss of generality we may assume
$${\rm d}_{\Gamma}(P_i)\geq \delta,\quad \bar P_i<\bar P_{i+1}\,\,\forall i,\quad {\rm d}_{\Gamma}(P_1)=\delta+\eta, \quad |P_i-P_j|\geq 2\delta \,\,\,\forall\ i\neq j.$$
Consider $\bar P_i$ the projections of $P_i$ onto $\gamma_{\delta}$, i.e.
 $$P_i=\bar P_i-\eta_i \nu_{\bar P_i},\;\;\bar P_i\in\gamma_{\delta},\,\,\eta_i\in [0,\eta].$$  If there exist $i\neq j$ such that  $|\bar P_i-\bar P_j|\leq \frac{\delta}{2}$, then,

$$\begin{aligned}|P_i-P_j|^2&=|\bar P_i-\eta_i\nu_{\bar P_i}-\bar P_j+\eta_j\nu_{\bar P_j}|^2\\ &=|\bar P_i-\bar P_j|^2+|\eta_i\nu_{\bar P_i}-\eta_j\nu_{\bar P_j}|^2-2\langle \bar P_i-\bar P_j,\eta_i\nu_{\bar P_i}
-\eta_j\nu_{\bar P_j}\rangle\\ &\leq |\bar P_i-\bar P_j|^2+4\eta|\bar P_i-\bar P_j| +4\eta^2 \leq \frac{\delta^2}{4}+2\delta\eta+4\eta^2< 4\delta^2,\end{aligned}$$ and so $\varphi_{k}(\p)<\delta$, by which  \eqref{thesis} follows.
Now assume  $|\bar P_i-\bar P_j|\geq \frac{\delta}{2}$ if $i\neq j$. Then $|\bar P_i-\bar P_j|\geq |P_i-P_j|$ by
\eqref{derlem2}. If  $|\bar P_i-\bar P_j|<2\delta$ for some $i\neq j$, then again $\varphi_{ k}(\p)<\delta$ and we have done. Now assume $|\bar P_i-\bar P_j|\geq2\delta$ for every $i\neq
j$, which means $\varphi_{ k}(\bar{\p})=\delta$. By Lemma \ref{lem1}   $|\bar P_{2}-\bar P_1|=2\delta$. Then
\eqref{derlem2} implies $|P_{2}-P_1|<|\bar P_{2}-\bar P_1|=2\delta$, by which $\varphi_{ k}(\p)<\delta$, and \eqref{thesis} follows. Combining \eqref{thesis} with Lemma \ref{lem1} we obtain the thesis.

 \end{proof}

\noindent{\bf Proof of Theorems \ref{th1} completed.}
Let us fix $\delta_0>0$  is sufficiently small such that for any $\delta\in(0,\delta_0]$  $\gamma_\delta$
is a  regular closed curve and satisfies \eqref{mimim}. Then, let  us take $\delta \in (0,\delta_0)$, $k\in \N$ even such that Proposition \ref{lemderlem2} holds for $\eta>0$ sufficiently small.
By \eqref{mimim} we deduce
$$\min\left\{\frac12\min_{j-i\geq 2\atop ( i,j)\neq (1,k) }|P_i-P_j|\,\Bigg|\, P_i\in \gamma_\delta,\; P_i<P_{i+1}\;\;\forall i,\; |P_i-P_{j}|\geq2 \delta\;\;\forall i\neq j\right\}>\mu>\delta$$ for some suitable $\mu>0$.
Hence, possibly reducing the number $\eta$, we may assume
$$\min\left\{\frac12\min_{j-i\geq 2\atop ( i,j)\neq (1,k)}|P_i-P_j|\,\Bigg|\, \p=(P_1,\ldots, P_k)\in U_\eta\right\}>\mu>\delta.$$
Now, if $\p=(P_1,\ldots, P_k)\in U_\eta$ and $i<j$ is such that $(-1)^{i+j}=1$, then $j-i\geq 2$ and, recalling also that $k$ is even,  $( i,j)\neq (1,k)$, consequently,
$|P_i-P_j|\geq \mu$.
By using Remark \ref{energy3} we get
\begin{equation*}\begin{aligned}M_\e [ \p]=&\, \sum_{i=1}^{k}e^{-\frac{2\dd(P_i)+o(1)}{\e}}+\sum_{i,j=1\atop i<j}^ke^{-\frac{|P_i-P_{j}|+o(1)}{\e}}+O(e^{-\frac{\mu}{\e}})
\\ &
=
e^{-\frac{2\varphi_k(\hbox{\scriptsize$\p$})+o(1)}{\e}}+O(e^{-\frac{\mu}{\e}})\end{aligned}
\end{equation*}
uniformly for ${\bf P }=(P_1,\ldots,P_k)\in
U_\eta$.
Proposition \ref{lemderlem2} applies and gives  $$M_\e [ \p^*]=e^{-\frac{2\delta+o(1)}{\e}}$$ where $\p^*\in U_\eta$ is such that $\varphi_k(\p^*)=\delta$,
and $$\inf_{\hbox{\scriptsize$\p$}\in\partial U_\eta}M_\e [ \p]\geq e^{-\frac{2\delta'}{\e}}$$
for some $\delta'<\delta$.
We conclude that $ M_\e$ has a minimum point $\p_\e=(P_1^\e,\ldots, P_k^\e) $ in $ U_\eta$.
According to Lemma \ref{relation} and Lemma \ref{validity}, for
$\e>0$ sufficiently small $ v_{\e}:=w_{\e,{\bf P}_\e}+ \phi_{\e,{{\bf P}}_\e}$ solves the problem \eqref{sch}.
 Finally if $(P_1^*,\ldots, P_k^*)$ is the limit of  a subsequence of $(P_1^\e,\ldots, P_k^\e)$, Proposition \ref{lemderlem2} implies  that $(P_1^*,\ldots, P_k^*)$ satisfies \eqref{pack}.

 Thus the thesis of Theorem \ref{th1} holds, \eqref{deluxe} following from \eqref{known0} and Lemma \ref{reg}.

\end{document}